\documentclass[reqno,11pt,a4paper]{amsart}
\usepackage{amsmath,amssymb,latexsym}

\usepackage{enumerate}
\usepackage[all]{xy}
\usepackage{fancyhdr}
\usepackage{color}

\usepackage{graphics}
\usepackage{hyperref}

\newcommand{\Z}{\mathbb Z}
\newcommand{\Q}{\mbox{$\mathbb Q$}}

\newcommand{\T}{\mathcal T}

\newcommand{\N}{\mathbb{N}}

\newcommand{\F}{\mathcal{F}}

\newcommand{\cyc}{\mathrm{cyc}}

\newcommand{\La}{\Lambda}

\newcommand{\lra}{\longrightarrow}

\usepackage[OT2,OT1]{fontenc}
\newcommand\cyr{%
\renewcommand\rmdefault{wncyr}
\renewcommand\sfdefault{wncyss}
\renewcommand\encodingdefault{OT2}
\normalfont\selectfont}

\DeclareTextFontCommand{\textcyr}{\cyr}

\newtheorem{theorem}{Theorem}[section]
\newtheorem{proposition}[theorem]{Proposition}
\newtheorem{lemma}[theorem]{Lemma}

\newtheorem{corollary}[theorem]{Corollary}
\newtheorem{definition}[theorem]{Definition}

\newtheorem*{theorem*}{Theorem}

\theoremstyle{definition}
\newtheorem*{example}{Example}

\theoremstyle{definition}
\newtheorem{remark}[theorem]{Remark}

\theoremstyle{definition}
\newtheorem*{dfn}{Definition}
\theoremstyle{plain}
\newtheorem*{namedthm}{\namedthmname}
\newcounter{namedthm}

\makeatletter
\newenvironment{named}[1]
{\def\namedthmname{#1}%
	\refstepcounter{namedthm}%
	\namedthm\def\@currentlabel{#1}}
{\endnamedthm}
\makeatother

\hypersetup{
   colorlinks=true, 
    linktoc=all,     
    linkcolor=blue,
    citecolor=blue,}

\usepackage[capitalise,nameinlink]{cleveref}

\newcommand{\Keywords}[1]{\par\noindent
{\small{Keywords and phrases}: #1}}

\newcommand{\AMS}[1]{\par\noindent
{\small{AMS Subject Classification}: #1}}

\author{Sohan Ghosh, Somnath Jha, Sudhanshu Shekhar}
\address{Department of Mathematics and Statistics, IIT Kanpur, Kanpur 208016, India}
\email{gsohan@iitk.ac.in, jhasom@iitk.ac.in, sudhansh@iitk.ac.in }

\begin{document}

\title{Twisting lemma for $\Lambda$-adic modules}

\begin{abstract}
A classical twisting lemma says that given a finitely generated torsion module $M$ over the Iwasawa algebra $\Z_p[[\Gamma ]]$ with $\Gamma \cong \Z_p, \ $there exists  a continuous character $\theta: \Gamma \rightarrow \Z_p^\times$ such that,  the $ \Gamma^{p^n}$-Euler characteristic of  the twist $M(\theta)$ is finite for every $n$. This twisting lemma has been generalized for the  Iwasawa algebra of a general compact $p$-adic Lie group $G$.  In this article, we consider a further generalization of the twisting lemma to $\T[[G]]$ modules, where $G$ is a compact $p$-adic Lie group and $\T$ is a finite extension of $\Z_p[[X]]$. Such modules naturally occur in  Hida theory. We also indicate arithmetic applications by considering the  `big' Selmer  (respectively fine Selmer) group   of a $\La$-adic form over a $p$-adic Lie extension.
 
\end{abstract}

\maketitle

\let\thefootnote\relax\footnotetext{

\AMS{11R23, 14F33, 11G05}
\Keywords {\begin{footnotesize}{Iwasawa theory,  Selmer groups, $\La$-adic form, $G$-Euler characteristic.}\end{footnotesize}}
}
\section*{Introduction}
In this article, we discuss some topics in non-commutative Iwasawa theory for modules over $\T[[G]]$; the completed group ring of a compact $p$-adic Lie group $G$ with coefficients in the ring $\T$, where $\T$ is a finite extension of $\Z_p[[X]]$.

We fix an odd prime $p$ throughout. Let $B$ be any commutative, complete, noetherian local domain of characteristic $0$ with finite residue field of characteristic $p$. For a  profinite group $\mathcal{G}$, recall the Iwasawa algebra of $\mathcal G$ over $B$ is defined as $B[[\mathcal{G}]]:=\underset{U}\varprojlim \ B[\mathcal{G}/U],
$
where $U$ varies over open normal subgroups of $\mathcal{G}$ and the inverse limit is taken with respect to the canonical projection maps.
 
 Let $G$ be a compact $p$-adic Lie group with a closed normal subgroup $H$ such that $\Gamma:=G/H\cong \Z_p$. {{}We will denote by $O$, the ring of integers of a finite extension of $\Q_p$.} For a left $O[[G]]$ module $M$ and a continuous character $\theta:\Gamma\rightarrow \Z_p^{\times}$,  denote by $M(\theta)$ the $O[[G]]$-module $M\otimes_{\Z_p} \Z_p(\theta)$ with diagonal $G$-action. We will assume throughout that $G$ has no $p$-torsion element. Recall that for a compact $p$-adic Lie group $G$ with no element of order $p$, the Iwasawa algebra $O[[G]]$ has finite global dimension \cite{B,la}. 

\begin{dfn}
	Let $G$ be a compact $p$-adic Lie group without any element of order $p$. For a finitely generated $O[[G]]$-module $M$, we say that the $G$-Euler characteristic of $M$ exists if the homology groups $H_i(G,M)$ are all finite and we define it as 
	\begin{equation*}
	\chi(G,M):=\underset{i\geq 0}\prod (\# H_i(G,M))^{(-1)^i}  \ .
	\end{equation*}	
\end{dfn}	
Given {{}an} $O[[G]]$-module $M$, $\chi(G,M)$ is an  invariant attached to $M$. A natural example of $O[[G]]$ modules in arithmetic comes from Selmer group attached to a motive over $p$-adic Lie extension of a number field, with $G$ being the corresponding Galois group. {The }Euler characteristic of {{} a } Selmer group naturally carries arithmetical information.  For any number field $K$, let $K_\cyc$ denote the cyclotomic $\Z_p$ extension and set $\Gamma=\text{Gal}(K_\cyc/K)$. Let $E/\Q$ be an elliptic curve with  good, ordinary reduction at $p $ and let $X(E/\Q_\cyc)$ denote the dual $p^\infty$-Selmer group  of $E$ over $\Q_\cyc$. Then it is known that, under suitable condition, the $p$-adic valuation of
$
  \chi(\Gamma,X(E/\Q_{\cyc})) $ is related to the $p$-adic valuation of the special value $\frac{L_E(1)}{\Omega_E}  (\# \tilde{E}(\mathbb F_p)(p))^2$ of the complex $L$-function of $E$ over $\Q$.  (see \cite[Introduction]{ss})
  
On the other hand, consider a $p$-adic Lie extension $K\subset K_\cyc \subset K_\infty$ of a number field $K$ which is unramified  outside a finite set of  places of $K$ and set $G = \text{Gal}(K_\infty/K)$ and $H = \text{Gal}(K_\infty/K_\cyc)$. Then $\Gamma := G/H \cong \Z_p$. For such a general $p$-adic Lie extension $K_\infty/K$,  the twisted Euler characteristic of $\Z_p[[G]]$ modules $M$, such that $M/{M(p)}$ is finitely generated over $\Z_p[[H]]$, has been discussed in \cite{cfksv}. Also for an elliptic curve $E/\Q$, the conjectural relation between $\chi(G, X(E/K_\infty)(\theta))$ and  twisted $L$-values are studied  (cf. \cite[Theorem 3.6]{cfksv}, \cite{ss}). 

For $G= \Gamma$, following classical twisting  lemma is well known  in Iwasawa theory and can be found in the works of Greenberg \cite{gr1} and Perrin-Riou \cite{pr}: For any finitely generated torsion $O[[\Gamma]]$-module $M$, there exists 
a continuous character $\theta : \Gamma \rightarrow \Z^\times_p$ such that the largest $\Gamma^{p^n}$-coinvariant quotient 
$H_0(\Gamma^{p^n}, M(\theta))=(M(\theta))_{\Gamma^{p^n}}$  is finite for every $n \in \N$. Note that $H_0(\Gamma^{p^n}, M(\theta))$ is finite if and only if 
$\chi(\Gamma^{p^n}, M(\theta))$  is finite.

Let $R$ be a ring and $M$ be a left $R$ module. Define, $M(p):=\underset{r\geq 1}\bigcup M[p^r]$,  where $M[p^r]$ is the set of $p^r$ torsion points of $M$.  Unless stated otherwise, when we consider a module $M$ over a ring $R$, we mean $M$ is a left module over $R$.

The twisting Lemma in the non-commutative setting was established  in \cite{joz}:
\begin{named}{Theorem(JOZ)}\label{thm2}
		Let $G$ be a compact $p$-adic Lie group and $H$ be a closed normal subgroup of $G$ such that $ \Gamma:=G/H\cong\Z_p$. Let $M$ be an $O[[G]]$ module which is finitely generated over $O[[H]]$. Then, there {exists} a continuous character $\theta:\Gamma\rightarrow \Z_p^{\times}$ such that $M(\theta)_U=H_0(U,M(\theta))$ is finite for every open normal subgroup $U$ of $G$.
\end{named}

 Note that for a general compact $p$-adic Lie group $G$ and {{}an} $O[[G]]$ module $M$, $H_0(G,M)$ is finite does not necessarily imply $\chi(G,M) $ exists (i.e. finite) \cite[see Remark 1.5]{ss}. In \cite{ss},  \ref{thm2} was extended to the following result on the twisted Euler characteristic.
\begin{named}{Theorem(JS)}\label{thm1}
	Let $G$ be a compact $p$-adic Lie group without any element of order $p$ and $H$ be a closed normal subgroup of {{} $G$} such that $\Gamma:=G/H\cong \Z_p$. Let $M$ be a finitely generated $O[[G]]$ module such that $M/M(p)$ is a finitely generated $O[[H]]$ module. Then there {{} exists} a continuous character $\theta:\Gamma\rightarrow \Z_p^{\times}$ such that $\chi(U,M(\theta))$ exists for every open normal subgroup $U$ of $G$.
\end{named}
Moreover,  it is shown that for a given $M$, there is a countable subset $S_M$ of all continuous characters    from $\Gamma$ to $\Z_p^{\times}$, such that  for any choice of  a continuous $\theta: \Gamma$ to $\Z_p^{\times}$ outside $S_M$,  \ref{thm2} and \ref{thm1} hold.

Congruence of modular forms is an important topic  in number theory and it naturally leads to the study of modules over {{}`}two variable{{}'} Iwasawa algebra { where the}  coefficient ring  {{} of the Iwasawa algebra} is a certain universal ordinary deformation ring (cf. \cite{hi}, \cite{w}). These {universal ordinary deformation} rings {{} are typically}  finite extensions of $\Z_p[[X]]$. Thus it is natural to ask for a generalization of twisting Lemma for modules over $\T[[G]]$, {where $G$ is a compact $p$-adic Lie group} and $\T$ is finite over $\Z_p[[X]]$. The main result of the article is the following:

\begin{theorem}\label{thm-main}
Let $G$ be a compact $p$-adic Lie group without any element of order $p$ and $H$ be a closed normal subgroup of $G$ with $\Gamma:=G/H\cong \Z_p$. Let $\T$ be  a commutative, complete local domain which is  finite over $\Z_p[[X]]$. Let $M$ be a finitely generated $\T[[G]]$ module such that $M/M(p)$ is   finitely generated over  $\T[[H]]$. 

Then {{} there exists} a continuous character $\theta:\Gamma \rightarrow \Z_p^{\times}$ and a countable set  $C_{M,\theta}$  of height $1$ prime ideals of $\T$ such that if we choose  any height $1$ prime ideal $Q \not\in C_{M,\theta}$, then 
$\chi\big(U,\frac{M}{QM}(\theta)\big)$ is finite for every open normal subgroup $U$ of $G$.

\end{theorem}

\begin{definition}\label{sandc}
Define  $S:=\{ \theta| \theta \text{ is a continuous character from }\Gamma\ { to }\ \Z_p^{\times} \}$ and set $C= C^\T:=\{Q: Q \text{ is a height 1 prime ideal of }\T, Q \nmid (p) \}$.
\end{definition}

\begin{remark}
In fact, in the proof of \cref{thm-main}, we will establish a slightly stronger result by showing  there exists a countable  subset $S_M$ of $S$ such that for any choice of  $\theta \in S\setminus S_M$, \cref{thm-main} holds.
\end{remark}
\begin{remark}\label{cfksv}
In   Theorem \ref{thm-main}, we consider finitely generated ${{}\T}[[G]]$ modules $M$ such that $M/M(p)$ is finitely generated over ${{}\T}[[H]]$. This comes from the formulation of non-commutative GL$_2$ Iwasawa main conjecture in \cite{cfksv}, where it is suggested, many arithmetic modules, including the dual Selmer group of an elliptic curve over a $p$-adic Lie extension would satisfy this condition.
\end{remark}

 \begin{remark}\label{rem1.6}
    	Let us keep the setting and hypotheses of Theorem \ref{thm-main}. Also for simplicity, take $\mathcal T = O[[X]]$. Then, from Theorem \ref{thm-main}, one may naturally ask {{}the following questions}:

{{}Question} 1: {{} Is it possible that} for {all but finitely many} height 1 prime ideals in $O[[X]]$, $\big(\frac{M}{QM}\big)_U=H_0\big(U,\frac{M}{QM}\big)$ is finite, for every open normal subgroup $U$ of $G$ {{}?}
    	
\medskip
    	
    	{{}Question} 2: {{} Does} there exist a continuous character $\theta:\Gamma\rightarrow \Z_p^{\times}$ and a {finite} set $C'$ of $C$ such that if we choose and fix any $Q \in C\setminus C'$ then $\left(\frac{M}{QM}(\theta) \right)_U$ is finite for every open normal subgroup $U$ of $G$ {{}?} 	
    	
    	\medskip
	
    	{{}Question} 3: {{} Does} there exist a {countable} subset  $S'$ of $S$ and a \underline{countable} subset $C'$ of $C$ such that  for any choice of $\theta\in S\setminus S'$ and for any choice of a height 1 prime $Q \in C \setminus C'$,  $\big(\frac{M}{QM}(\theta)\big)_U$ is finite for  every open normal subgroup $U$ of $G$ {{}?}	  
    	
    	\medskip
	
    	{{}Question} 4: {{} Does} there exist a {countable} set $S'\times C' \subset S \times C$,  such that for any choice of $(\theta, Q) \not\in S'\times C'$,  $\big(\frac{M}{QM}(\theta)\big)_U$ is finite, for every open normal subgroup $U$ of $G$ {{}?}
	\end{remark}
  {{}
   
 Let $S'$ and $C'$ be countable subsets of $S$ and $C$ respectively. We observe  that $(S\setminus S')\times(C\setminus C')\subset (S\times C)\setminus(S'\times C')$. From this it follows  that a negative answer to question 3 gives a negative answer to question 4. 
 
For questions 1 and 2,  we will find  $O$, $G$, $H$, an $O[[X]][[G]]$ module $M$ and  for every $\theta\in S$, a countably infinite subset $C'$ of $C$ such that for and every $Q \in C'$,  $\big(\frac{M}{QM}(\theta)\big)_{U_0}$ is infinite for some open normal subgroup $U_0$ of $G$. This will provide  a negative answer to both question 1 and question 2.  Further, we will show the same module $M$ will give a negative answer to question 3 (and hence question 4).

 }
 
 \medskip
   
    \begin{example} \label{ex1}
    Let $G=\Gamma =< \gamma >$, $H=\{1\}$ and $O=\Z_p$.
	Let $M=\dfrac{\mathbb{Z}_p[[X]][[T]]}{(X-T)} \ ,$ where $1+T$  corresponds to $\gamma$. 
	Let $Q \neq (p)$  be a height 1 prime ideal of $\Z_p[[X]]$.  Write  $Q=(g(X))$, where  $(g(X))$ is an irreducible Weierstrass polynomial. 
	
	For a continuous character  $\theta: \Gamma \lra \Z_p^{\times}$, we can write $\theta(\gamma^{-1})=1+p\lambda$ for some (fixed) $\lambda\in \Z_p$. Then
	\begin{small}{
			$M(\theta)\cong \dfrac{\mathbb{Z}_p[[X]][[T]]}{((1+p\lambda)(T+1)-(X+1))}.$}\end{small}
	Now, if $Q \neq (p)$ is  a prime ideal of height 1 in $\mathbb{Z}_p[[X]]$, then
	\begin{small}{\begin{equation*}
				\big(\frac{M}{QM}(\theta)\big)_{\Gamma^{p^n}}\cong\frac{\Z_p[[X]][[T]]}{((1+p\lambda)(T+1)-(X+1),(T+1)^{p^n}-1,Q )}.
	\end{equation*}}\end{small}
	is finite if and only if   $(Q,(X+1)^{p^n}-(1+p\lambda)^{p^n})$ has height 2. Let
	\begin{equation*}
		C_{p^n}^{\theta}: =\{Q\text{ is a height 1 prime in}\ \Z_p[[X]]\	:\ Q\neq(p) \text { and} \ \big(\frac{M}{QM}(\theta)\big)_{\Gamma^{p^n}} \text{ is infinite}\}.
	\end{equation*}
	
	Then $C_{p^n}^{\theta}=\{(g(X)) \ | \ g(X) \ \text{is an irreducible divisor of}\ (X+1)^{p^n}-(1+p\lambda)^{p^n} \}.$  For $m>n$, by Weierstrass preparation theorem, there exists an irreducible factor $\ell^\theta(X)$ of $\frac{(X+1)^{p^m}-(1+p\lambda)^{p^m} }{(X+1)^{p^n}-(1+p\lambda)^{p^n}  }$. Then, $(\ell^\theta(X))\in C_{p^m}\setminus C_{p^n}$ and hence  for $m>n$, $C_{p^n}^{\theta}\subsetneq C_{p^m}^{\theta}$ and $C'_\theta:=\bigcup\limits_{n\geq1} C_{p^n}^{\theta}$ is an infinite set. Now, for each  $Q\in C'_\theta$,  there exists  an integer $n_0$ such that $\big( \frac{M}{QM}(\theta)\big)_{\Gamma^{p^{n_0}}} $  is infinite.
	This shows that {{} the answer to question} 2  in remark \ref{rem1.6} is {{} negative}.

	 {{} Now} recall the sets $S$  and $ C$ defined in definition \ref{sandc}.  If possible, assume that {{} the answer to question} 3 in Remark \ref{rem1.6} is true.  Then, there exists a countable subset $S'$ of $S$  and $C'$ of $C$ such that for $\theta \in S\setminus S'$ and $Q\in C\setminus C'$, $H_0(\Gamma^{p^n},\frac{M}{QM}(\theta))$ is finite for every $n\geq 0$.

		For every $\theta\in S$, we can write   $\theta(\gamma^{-1})=1+p\lambda_\theta$ for some (fixed) $\lambda=\lambda_\theta \in \Z_p$. Then $\big(\frac{M}{QM}(\theta)\big)_{\Gamma^{p^n}} \cong\frac{\Z_p[[X]]}{((X+1)^{p^n}-(1+p\lambda)^{p^n}, Q)}$. Observe that if $Q= (X-p\lambda)$ then $\big(\frac{M}{QM}(\theta)\big)_{\Gamma^{p^n}}$ is infinite for every $n \in \N$.  
		
		Now let us choose a  $\theta\in S\setminus S'$. Then by our assumption in Statement 3, $(\frac{M}{QM}(\theta))_{\Gamma^{p^n}}$ is finite for every $Q\in C\setminus C'$. Thus the height 1 prime $(X-p\lambda)$ must be in $C'$. Hence for every $\theta\in S\setminus S'$, we get an element $Q_{\theta}=(X+1-\theta(\gamma))\in C'$. Note  for $\theta_1, \theta_2\in S\setminus S'$ with $\theta_1\neq \theta_2$, $Q_{\theta_1}\neq Q_{\theta_2}$. Since $S\setminus S'$ is uncountable, the set $C'$ is also uncountable. This is a contradiction and hence question 3 has a negative answer as well.
 	   \end{example}

      \medskip
      
  However, the following variant of the Theorem \ref{thm-main} holds true.
\begin{proposition}\label{th1.22}
 	 Let $G, H, \T$ and $M$  be as in Theorem \ref{thm-main}. Let $C'$ be a  countable subset of $C$. Then there exists a countable set $S_M $ of $S$, such that for any choice of $\theta \in S\setminus S_M$ and for any choice of $Q \in C'$, $\chi(U,\frac{M}{QM}(\theta))$ exists, for every open normal subgroup $U$ of $G$.
\end{proposition}
\begin{proof}
Let  us enumerate $C'= \{Q_1, Q_2, \cdots\}$.  For each $Q_i$, note that $\frac {\T}{Q_i\T}[[G]] \cong {O_i}[[G]]$ and $\frac {\T}{Q_i\T}[[H]] \cong {O_i}[[H]]$, where $O_i$ is the ring of integers of a finite extension $L^{i}$ of $\Q_p$. Then from  \ref{thm1},  there exists a countable subset $S_{Q_i, M}$ of all the continuous character from $\Gamma$ to $\Z_p^{\times}$, such that for any choice of  $\theta\notin S_{Q_i, M}$, $\chi(U,\frac{M}{Q_i M} (\theta))$ exists for every open normal subgroup $U$ of $G$. Now choose and fix any $\theta : \Gamma\rightarrow \Z_p^{\times}$ outside the countable set $S_M:=\bigcup\limits^{{{}\infty}}_{i=1} S_{Q_i,M} $. Then it follows that for any choice of $Q_i \in C'$, $\chi(U,\frac{M}{Q_i M} (\theta))$ exists for every open normal subgroup $U$ of $G$. \end{proof}	

Let $\mu_{p^\infty}$ denote the group of $p$ power roots of unity. Define a homomorphism $\nu_{k,\zeta}:\Z_p[[X]]  \rightarrow \bar{\Q}_p^{\times}$ by $\nu_{k,\zeta}(1+X)= \zeta (1+p)^k$, where $\zeta\in \mu_{p^\infty}$ and $k\in \N$.
Define
$A_{\text{arith}} (\Z_p[[X]]):=\{Q| Q=\text{ker}(\nu_{k,\zeta}) \text{ for some } k\geq1\ \text{and some}\ \zeta\in \mu_{p^\infty} \}$. Clearly, elements of $A_{\text{arith}} (\Z_p[[X]])$ are height 1 prime ideals of $\Z_p[[X]]$. Similarly, we define $A_{\text{arith}}(\T)$ as {{} the} set of those height 1 prime ideals of $\T$ which {{}divide} some height 1 prime ideal in $A_{\text{arith}} (\Z_p[[X]]) $. The elements of $A_{\text{arith}}(\T)$ are called the arithmetic primes or classical primes of $\T$ [see \cite{w}]. As an immediate corollary of Proposition \ref{th1.22}, we deduce the following:
\begin{corollary}\label{col11}
	Let us keep the hypotheses and setting of Proposition \ref{th1.22}.
	Then, {{} there  exists} a countable subset $S_M$ of $S$, such that for any choice of  $\theta\in S\setminus S_M$ and for any $Q\in  A_{\text{arith}}(\T)$, $\chi(U,\frac{M}{QM}(\theta))$ exists for every open normal subgroup $U$ of $G$.
\end{corollary}
\subsection*{Significance in Arithmetic}\label{application0}
We set up a necessary framework that is  needed to describe  applications of our results in arithmetic. Consider  a $p$-adic Lie extension $K\subset K_\cyc\subset K_\infty$ of a number field $K$, such that $ K_\infty$  is  unramified outside finitely many places of $K$ and  $G=\text{Gal}(K_\infty/K)$ has no element of order $p$. Put  $H=\text{Gal}(K_\infty/K_\cyc)$ and note $\Gamma:=G/H=\text{Gal}(K_\infty/K)\cong \Z_p$.  For any open normal subgroup $U$ of $G$,  the fixed field $K_U:=K_\infty^U$ defines a finite extension of $K$ inside $K_\infty$.

Let $f$  be a $p$-ordinary newform of weight $\geq 2$. We denote by  $L_f$, a lattice of $V_f$, the $p$-adic Galois representation associated to $f$ and $O_f$ will denote the ring of integers of the $p$-adic field defined by the Fourier coefficients of $f$. 

Let   $ \T=\T_\mathcal F$ be the quotient of the universal
ordinary Hecke algebra   that
corresponds to an ordinary $\Lambda$-adic newform $\mathcal F$  (see
\cite{hi}). 
The algebra $\T_\F $  is a local domain  and is finite flat over $\Z_p[[X]]$. By a
celebrated result of Hida (\cite{hi}, \cite{w}), there  exists a
`large' continuous irreducible representation 
$\rho_{\F} : {G}_{\Q} \rightarrow  \mathrm{Aut}_{\T_\F}(\mathcal L_{\F}),$
where $\mathcal L_{\F}$ is a finitely generated, torsion-free module of generic rank 2 over 
$\T_\F$.  We assume that the residual representation $\bar{\rho}_\F$ associated to ${\rho}_\F$ (see \cite{hi}) is absolutely irreducible. For each $Q \in A_{\text{arith}}(\T_\F),$ {{} there  exists}   a  $p$-ordinary, $p$-stabilized newform  $f_{Q}$, such
that the quotient $\mathcal L_\F/{Q\mathcal L_F}$ is isomorphic to the (unique) lattice ${L}_{f_Q}$. To simplify the notation, we will write $L_Q= L_{f_Q}$ and $O_Q=O_{f_Q}$.

In this setting of $p$-ordinarity,  the dual Selmer group $X(L_f/K_\infty)$   \cite[Definition 1.11]{jo} (respectively $\mathcal X(\mathcal L_\F/K_\infty)$ \cite[ \S 3]{j}) of $f$ (respectively $\F$) over $K_\infty$ is defined and is a finitely generated $O_f[[G]]$ module (respectively $\T[[G]]$ module). Similarly, with  $f$ as above, for the twisted Galois representation $V_f(\theta)$, the dual Selmer group $X(L_f(\theta)/K_\infty)$  is also defined  and it is  a finitely generated $O_f[[G]]$ module. A control theorem, originally proved by Mazur, is a widely used tool in Iwasawa theory. In this setting, under suitable  condition, we deduce by a control theorem, that the kernel and cokernel of $X(L_f(\theta)/K_\infty)_U \lra  X(L_f(\theta)/K_U)$ are finite for every $U$ (see \cite[Theorem 0.1]{jo}). 

{\bf (I) Finiteness of the twisted Selmer groups over $K_U$:} Assume  that $\frac{X(L_f/K_\infty)}{X(L_f/K_\infty)(p)}$ is finitely generated over $O_f[[H]]$. Then \ref{thm2} along with a control theorem  shows {{} there  exists,} $ \theta: \Gamma \rightarrow \Z_p^\times $ such that the twisted Selmer group $X(L_f(\theta)/K_U)$  is finite for every finite extension $K_U$ inside $K_\infty$. 

Now, our Corollary \ref{col11} in particular shows that {{} there  exists,} $ \theta \in S$ such that for each $Q \in A_{\text{arith}}(\T)$, $\big(\frac{M}{QM}(\theta)\big)_U$ is finite for every $U$. Take $M =\mathcal X(\mathcal L_\F/K_\infty)$ and assume $\frac{\mathcal X(\mathcal L_\F/K_\infty)}{\mathcal X(\mathcal L_\F/K_\infty)(p)}$ is a finitely generated $\T_\F[[H]]$  module. Then by estimating the kernel and cokernel of $\frac{\mathcal X(\mathcal L_\F/K_\infty)}{Q\mathcal X(\mathcal L_\F/K_\infty)}\rightarrow  X( L_Q/K_\infty)$ \cite[\S3]{j}, we can deduce the following: For any chosen $ \theta \in S \setminus S_{\mathcal X(\mathcal L_\F/K_\infty)}$,  the  twisted Selmer groups $X(L_Q(\theta)/K_U)$ of the congruent family of cuspforms $f_Q$,  $Q \in A_{\text{arith}}(\T)$ are all simultaneously finite when $K_U$ varies over every finite extension   of $K$ inside $K_\infty$.

{\bf Algebraic functional equation:} For a given $f$, the finiteness of $X(L_f(\theta)/K_U)$ for a fixed $\theta$, where  $K_U$ varies over every  intermediate finite extension of $K$ inside $K_\infty$,  can be useful in various arithmetical situations; for example,  finiteness of $X(L_f(\theta)/K_U)$ is crucial in the proof of algebraic functional equation for $X(L_f/K_\infty)$ (see \cite[\S 5, Theorem 0.3]{jo}).  \ref{thm1} was also used in the proof in \cite{jo}. Indeed, this is a generalization of the fact that classical twisting lemma is crucial in Greenberg's \cite[Theorem 2]{gr1} and Perrin-Riou's \cite[Theorem 4.2.1]{pr} proof of algebraic functional equation. 

Similarly, the  finiteness of  $X(L_Q(\theta)/K_U)$ where   $Q$ and $U$ both vary, can be used, for example, to establish an algebraic functional equation for the `big' Selmer group  $\mathcal X(\mathcal L_\F/K_\infty)$ of $\F$ over $K_\infty$.

{\bf (II) Fine Selmer group:}  The dual fine Selmer group  $Y(L_f/K_\infty)$ (resp. $\mathcal Y(\mathcal L_\F/K_\infty)$)  of  $f$ (resp. $\F$) over $K_\infty$, is a quotient of  $X(L_f/K_\infty)$ (resp. $\mathcal X(\mathcal L_\F/K_\infty)$) \cite[\S 2]{j}) and is a finitely generated $O_f[[G]]$ (resp. $\T[[G]]$) module. It has  interesting arithmetic properties. Following  conjectures of Coates-Sujatha  \cite[Conjectures A \& B]{cs}, it is believed that $Y(L_f/K_\cyc)$ is a finitely generated $\Z_p$ module and if $\text{dim}\ G>1$, then $Y(L_f/K_\infty)$ is a pseudonull $O_f[[G]]$ module, for any $f$. For $\text{dim}\ G>1$, it is not easy to determine if $\chi(G, Y(L_f/K_\infty))$ exists, even if we assume $Y(L_f/K)$ is finite and not much is known  in the literature.  Thus it is reasonable to discuss, if at least after a twist, the Euler characteristic of the fine Selmer group exists and  also makes sense to consider the Euler characteristic of a twist of the fine Selmer groups in a congruent family.
 \begin{corollary}\label{last}
 Assume $\frac{\mathcal Y(\mathcal L_\F/K_\infty)}{\mathcal Y(\mathcal L_\F/K_\infty)(p)}$ is a finitely generated $\mathcal T_\mathcal F[[H]]$ module. Then applying Corollary \ref{col11},  we deduce that {{} there  exists} a countable subset $S_{\mathcal Y(\mathcal L_\F/K_\infty)}$ of $S$ such that for any chosen $\theta \in S\setminus S_{\mathcal Y(\mathcal L_\F/K_\infty)}$ and for every  $Q \in A_{\text{arith}}(\T_\F)$, $\chi(U, \frac{\mathcal Y(\mathcal L_\F/K_\infty)}{Q\mathcal Y(\mathcal L_\F/K_\infty)}(\theta))$ is finite for every open normal subgroup $U$ of $G$.
\end{corollary}
Let $m$ be any $p$-power free integer. Now consider a particular example of a $p$-adic Lie extension (`false Tate curve' extension) given by $K =\Q(\mu_p)$ and  $K_\infty =\underset{n}{\cup}\Q(\mu_{p^\infty})( m^{1/p^n})$.  Assume that $\T_\F=O[[X]]$.  Then  the kernel and the cokernel of the natural map $\frac{\mathcal Y(\mathcal L_\F/K_\infty)}{Q\mathcal Y(\mathcal L_\F/K_\infty)}\lra  Y( L_Q/K_\infty)$ are finitely generated $O_Q$ modules \cite[Remark 11]{j}. Also assume for some $Q_0 \in A_{\text{arith}}(\T_\F), \ $ $Y( L_{Q_0}/\Q(\mu_{p^\infty}))$  is a finitely generated $O_{Q_0}$ module. Then $ \mathcal Y(\mathcal L_\F/K_\infty)$ and   $Y( L_Q/K_\infty)$ are finitely generated modules respectively over $\mathcal T_\mathcal F[[H]]$ and  $O_Q[[H]]$, for every $Q \in  A_{\text{arith}}(\T_\F)$ \cite[\S 2]{j}. From the proof of Proposition \ref{th1.22} and Corollary \ref{last}, we obtain the following result  on the family of fine Selmer groups $\{Y(L_Q/K_\infty): {Q\in A_{\text{arith}}(\T_{\mathcal F})\}}$.
\begin{theorem}\label{thmlast}
Let  $K_{\infty}/\Q(\mu_p)$ be the false Tate curve extension.  Let $\mathcal F$ be a $\Lambda$-adic newform. Assume  that  $\T_{\mathcal F}\cong O[[X]]$ is a power series ring and  {{} there  exists} $ Q_0 \in A_{\text{arith}}(\T_\F)$ such that $Y( L_{Q_0}/\Q(\mu_{p^\infty}))$  is a finitely generated $O_{Q_0}$ module. 

Then there exists a countable subset $S_{\mathcal Y(\mathcal L_\F/K_\infty)}$ of $S$, such that for any chosen $\theta \in S\setminus S_{\mathcal Y(\mathcal L_\F/K_\infty)}$ and for every $f_Q$,  $Q\in A_{\text{arith}}(\T_{\mathcal F})$,  the $U$-Euler characteristic $\chi(U, Y(L_Q/K_\infty)(\theta))$ exists, for every open normal subgroup $U$ of $G$. \qed
\end{theorem}

\begin{remark}
A  result similar to Corollary \ref{last}  and Theorem \ref{thmlast} respectively  for the usual `big' dual Selmer group $\mathcal X(\mathcal L_\F/K_\infty)$   (see \cite[\S 3]{j}) and for the Selmer group  $X(L_Q/K_\infty)$ of $f_Q$ for $\ Q \in A_{\text{arith}}(\T_\F)$ can also be obtained. 
\end{remark}

{\bf (III) $\Z_p[[H]]$ free Selmer group:} 
Let $E$ be the elliptic curve  of conductor $121$ defined by the equation $y^2 + y = x^3 - x^2-887x - 10143$. Let us consider a particular false Tate curve extension by taking $p=5$ and $m=11$.  Then $G\cong \Z_5 \rtimes \Z_5$.  Let $X(E/K_\infty)$  (respectively $X(E/K_\cyc)$) denote the  dual of the $5^\infty$-Selmer group of $E$ over $K_\infty$ (respectively $K_\cyc$). Then using the fact that $X(E/K_{cyc})$ has  $\mu$-invariant zero together with \cite[Theorem  3.1]{hv},  {{} there  exists}  an injective $\mathbb{Z}_5[[H]]$  module homomorphism $f : X(E/K_\infty)  \lra  (\mathbb{Z}_5[[H]])^t$, for some integer $t$. 
{{}In fact, using $\tilde{E}(\mathbb F_5)[5]=0$, from the proof of  \cite[Theorem  3.1]{hv}   we can show  that $f$ is an isomorphism. 
We explain it here briefly. From  \cite[Theorem  3.1(ii)]{hv}, $A:=\mathrm{cokernel}(f)$ is a $\Z_5[[H]]$ module of finite cardinality. Taking coinvariance by $H$, we get  that $H_1(H,A)$ is a subgroup of $X(E/K_\infty)_H$. Now it follows from \cite[equation 3.6]{hv}  and the assumption $\tilde{E}(\mathbb F_5)[5]=0$ that  $X(E/K_\infty)_H$ is a free $\Z_p$-module. This implies that $H_1(H,A)=0$. Since, $A$ is a finite  and $H$ is pro-cyclic, we deduce $H_0(H,A)=0$. Hence, by Nakayama's Lemma, $A=0$.  Therefore $X(E/K_\infty)$ is a free $\Z_p[[H]]$-module. }   

 Let  $Q_0 \in A_{\text{arith}}(\T_\F)$ be such that the weight 2 specialization $f_{Q_0}$ corresponds to  $E$ via modularity. Since $E$ has CM, in this case $\T\cong \Z_5[[X]]$. Then following an argument similar to \cite[page 412]{sh}, we can deduce that the natural map 
$\mathcal{X}(\mathcal{L}_\mathcal{F}/K_\infty)/{Q_0} \lra X(E/K_\infty)$ is an isomorphism.   
 Now it is well known that $ \mathcal{X}(\mathcal{L}_\mathcal{F}/K_\infty)$  has no non-trivial pseudonull  $\T[[G]]$ submodule.  In particular, $\mathcal{X}(\mathcal{L}_\mathcal{F}/K_\infty)[Q_0]=0$. Since $X(E/K_\infty)$ is a free $\Z_5[[H]]$-module of finite rank, by Nakayama's lemma  $\mathcal{X}(\mathcal{L}_\mathcal{F}/K_\infty)$ is also a free $\mathcal{T}[[H]]$-module of finite rank. This also implies that $\mathcal{X}(\mathcal{L}_\mathcal{F}/K_\infty)/Q \cong X(L_Q/K_\infty)$ is free $O_Q[[H]]$-module of finite rank for every  $Q \in  A_{\text{arith}}(\T_\F)$. We notice  that as the rank of  $E(\Q) \ =1$, the $U$-Euler characteristics of $X(E/K_\infty)$ does not exist for any open normal subgroup $U$ of $G$. 
 
 Write $M= \mathcal{X}(\mathcal{L}_\mathcal{F}/K_\infty)$. Then applying Theorem \ref{thm-main}, we get  a countable subset $S_M$ of $S$ and for each  $ \theta \in S \setminus S_M, $ {{} there  exists} a countable subset $C_{M,\theta}$ of $C$ such that the following holds: For every $ \theta \in S \setminus S_M$ and for any $Q \in C \setminus C_{M,\theta}$, the $U$-Euler characteristic $\chi(U, \frac{M}{QM}(\theta))$ exists for every $U$. Further, by Lemma  \ref{le1.4},  it is given by $\chi\big(U, \frac{M}{QM}(\theta)\big) = \# \big(\frac{M}{QM}(\theta)\big)_U$.
 
 Now consider $Q \in  A_{\text{arith}}(\T_\F)$. Then by Corollary \ref{col11},  {{} there  exists} a countable subset $S_{M}$ of $S$ such that for any   $ \theta \in S \setminus S_{M}$   and for  every $Q \in  A_{\text{arith}}(\T_\F)$, $\chi(U, \frac{M}{QM}(\theta))$ exists for every $U$. Moreover, by Lemma  \ref{le1.4}, it is given by $\chi\big(U, \frac{M}{QM}(\theta)\big)  = \#\big(\frac{M}{QM}(\theta)\big)_U  =\#\big(X(L_Q/K_\infty)(\theta)\big)_U    =\#\big(X(E/K_\infty)(\theta)\big)_U$. In particular, for every $Q \in A_{\text{arith}}(\T_\F)$,   $\chi(U, X(L_Q/K_\infty)(\theta)\big)$ is finite and $=\#\big(X(E/K_\infty)(\theta)\big)_U$, for every open normal subgroup $U$ of $G$.

\medskip

Now, we  give  a proof of Theorem \ref{thm-main}. Also see   Remarks \ref{lastint}  and \ref{sfinal-remark}. 
\section{Proof of Theorem \ref{thm-main}}\label{sec0}

 This proof of Theorem \ref{thm-main} is divided {{} into} several remarks, lemmas and propositions. We begin with various {{}reduction} steps.

\begin{remark}\label{rem101}
Let $G, H, \T$  be as in Theorem \ref{thm-main}. Let $N$ be a finitely generated $\T[[G]]$ module. Then consider the module $N(p)$. Note that $\T[[G]]$ is noetherian and hence  $N(p) =N[p^r]$ for some $r\in \N$. Thus from the short exact sequence $0\lra N[p^r] \lra N \lra \frac{N}{N[p^r]} \lra 0$, we get another right exact sequence \begin{small}{$$ \frac{N[p^r]}{QN[p^r]} \lra \frac{N}{QN} \lra \dfrac{\frac{N}{N[p^r]}}{Q\frac{N}{N[p^r]}}  \lra 0,$$}\end{small}
for any  $Q \in C$. Now let $I_{N,Q}$ be the image of $\frac{N[p^r]}{QN[p^r]}$ in $\frac{N}{QN} $. Then $I_{N,Q}$ is a finitely generated $p^r$-torsion $\frac{\T}{Q\T}[[G]]$ module and hence $\chi(U, I_{N,Q})$ exists for every $U$  \cite[Proposition 1.6]{how}. Thus, $\chi(U,\frac{N}{QN})$ is finite if and only if \begin{small}{$\chi\Big(U,\frac{\frac{N}{N[p^r]}}{Q\frac{N}{N[p^r]}}\Big)$}\end{small} is finite.  

We can apply  this observation for $N=M(\theta)$, with $M, \theta$ as in Theorem \ref{thm-main} and     without any loss of generality,  we may assume that $M$ in Theorem \ref{thm-main} is a finitely generated $\T[[H]]$ module.
\end{remark}
Next we consider finitely generated $\T[[G]]$ modules which are $\T$-torsion.

\begin{proposition}\label{ttorsion}
Let $G, H$ be as in Theorem \ref{thm-main}. Let $N$ be a finitely generated $\T[[G]]$ module such that $N$ is also $\T$ torsion. Also assume $\T$ is a regular (local) ring. 
Then there exists a finite subset $C_{N}$ of $C$ such that for  any $Q\in C \setminus C_{N}$, $\chi(U, \frac{N}{QN})$ exists for every open normal subgroup $U$ of $G$.
\end{proposition}

\proof  By Remark \ref{rem101}, we may assume $N$ is $p$-torsion free. {Let $N$ be generated by the elements $x_1,\cdots, x_n$ as a $T [[G]]$-module}. As $N$ is $\T$ torsion, {{} there  exists}$ \ r_i \in \T$ such that $r_ix_i=0$. Set $r:=r_1r_2....r_n$ and notice that $r$ being an element of  $\T$ commutes with the elements of $\T[[G]]$. Thus we get that $rN=0$. As $\T$ is a UFD, let $r=p_1^{n_1}p_2^{n_2}...p_t^{n_t}$ be the unique factorization of $r$ where $p_i, \ 1 \leq i\leq t$ are height 1 primes in $\T$. Set $C_N:=  \{p_1, p_2,...p_t\}$. Then for any $Q \in C\setminus C_N$, $\frac{N}{QN}$ is annihilated by the height 2 ideal $(r, Q)$. Now $\T$ is a regular local ring of dimension $2$. Thus for some $n_0 \in \N$, $p^{n_0} \in (r,Q)$. This shows $p^{n_0}\frac{N}{QN}=0$. Hence for any choice of $Q \in C\setminus C_N$, $\chi(U, \frac{N}{QN})$ exists  for every $U$ \cite[Prop. 1.6]{how}. \qed

\begin{corollary}\label{col1}
Let us keep the setting of  Proposition \ref{ttorsion}. Then in the proof  of Theorem \ref{thm-main}, we may assume that $M$ is $\T$-torsion free. 
\end{corollary}
\proof Take any $\theta \in S$ and write $N=M(\theta)$. Let $N_{\T}$ denote the $\T$-torsion submodule of $N$. Consider the exact sequence $0 \rightarrow N_{\T} \lra N \lra N/{N_{\T}}\rightarrow 0.$ Then for any $Q\in C$, we get the induced exact sequence \begin{small}{$$0 \rightarrow \frac{N_{\T}}{QN_{\T}} \lra \frac{N}{QN} \lra \dfrac{\frac{N}{N_{\T}}}{Q\frac{N}{N_{\T}}}\rightarrow 0.$$}\end{small} By Proposition \ref{ttorsion}, there exists a finite subset $C_N= C_{M,\theta}$ of $C$ such that for any $Q\in C\setminus C_N$, $\chi(U, \frac{N_{\T}}{QN_{\T}})$ exists for every $U$. Thus the statement of Theorem \ref{thm-main} holds for $M$ if and only if it holds for $M/{M_\mathcal T}$.\qed

Recall the following well known result, which can be conveniently found in \cite[Lemma 6.15]{sud}.
\begin{lemma}\label{hi}
Let $\T$ be  a  commutative, complete local domain  and let $\T$ be  finite over $\Z_p[[X]]$. Let $Q$ be a height 1 prime ideal in $\T$. Put $\mathfrak q= Q\cap \Z_p[[X]]$. Let  $Q=Q_1, Q_2,\cdots Q_d$ be the   height $1$ prime ideals in $\T$ lying above $\mathfrak q$. If $q$ is unramified in $\T$, then  the  kernel  and cokernel of the natural map $\T/{\mathfrak q \T} \lra \underset{1\leq i\leq d}{\bigoplus}\T/{Q_i\T}$ are finite of $p$-power order. \qed
\end{lemma}
Using Lemma \ref{hi}, we further reduce Theorem \ref{thm-main} to the case $\T=Z_p[[X]]$.
\begin{corollary}\label{hidaredcor}
  Let us keep the hypotheses and setting of Theorem \ref{thm-main}. Then in Theorem \ref{thm-main}, without any loss of generality, we may assume $\T=\Z_p[[X]]$.
\end{corollary}
\proof 
 Since the K\"alher differential $\Omega_{\T/{\Z_p[[X]]}}$ is supported outside finitely many height 1 prime ideals in $\Z_p[[X]]$, it follows that only finitely many height 1 primes of $\Z_p[[X]]$ can ramify in $\T$. Let $C^{\Z_p[[X]], \T}_1$ be the finite  set of height 1 prime ideals in $\Z_p[[X]]$ that ramify in $\T$. Let $\mathfrak q \in C^{\Z_p[[X]]} \setminus C^{\Z_p[[X]],\T}_1$ and let $Q=Q_1,Q_2,\cdots Q_d$ be the height 1 primes in $\T$ dividing $\mathfrak q$. Recall $M$ in Theorem \ref{thm-main} is a finitely generated $\T[[G]]$ module. By Lemma \ref{hi}, for any $\theta\in S$, there exists an exact sequence  
 \begin{small}{
\begin{equation}\label{eqredhida} 0 \rightarrow K^\theta_\mathfrak q \lra \frac{M(\theta)}{\mathfrak q M(\theta)} \lra \underset{1\leq i\leq d}{\bigoplus}\frac{M(\theta)}{{Q_i}M(\theta)} \lra  CK^\theta_\mathfrak q\rightarrow 0,
\end{equation}}\end{small} 
where $K^\theta_\mathfrak q$ and $CK^\theta_\mathfrak q$ are finitely generated $p$-power torsion $\frac{\T}{\mathfrak q\T}[[G]]$ modules. In particular, $K^\theta_\mathfrak q$ and $CK^\theta_\mathfrak q$ are finitely generated $p$-power torsion $\Z_p[[G]]$ modules. Now, as explained in Remark \ref{rem101}, $\chi(U, K^\theta_\mathfrak q)$ and $\chi(U, CK^\theta_\mathfrak q)$ always exist for any open normal subgroup $U$ of $G$.

Now assume  Theorem \ref{thm-main} holds for $\T = \Z_p[[X]]$. Then there exists  a countable subset $S_M$ of $S$ and for any  $\theta \in S\setminus S_M$ there exists a countable subset $C^{\Z_p[[X]]}_{M, \theta}$ of $C^{\Z_p[[X]]}$ such that the following holds:  For any  $\theta \in S\setminus S_M$ and for any $Q\in C^{\Z_p[[X]]}\setminus C^{\Z_p[[X]]}_{M, \theta}$, the $U$-Euler characteristic $\chi\left(U,\frac{M}{\mathfrak q M}(\theta)\right)$ exists for every open normal subgroup $U$ of $G$. Note $C^{\Z_p[[X]]}_{M,\theta} \cup C_1^{\Z_p[[X]], \T}$ is countable and hence $C^0_{M(\theta)}:= \{ Q \in C^\T : Q\cap \Z_p[[X]] \in C^{\Z_p[[X]]}_{M,\theta} \cup C_1^{\Z_p[[X]],\T}\}$ is a countable subset of $C^\T$. Now from \eqref{eqredhida}, for any $\theta \in S\setminus S_M$ and for any $Q\in C\setminus C^0_{M,\theta}$, $\chi(U, \frac{M}{Q M}(\theta))$ is finite for every $U$. Thus Theorem \ref{thm-main} holds for a general $\T$. \qed
\begin{remark}\label{rem102}
As explained in \cite[Remark 2.3]{ss}, we may also assume without any loss of generality, that $G$ is a compact, pro-$p$, $p$-adic Lie group without any element of order $p$. It should be noted \cite[Remark 2.3]{ss} uses fundamental work of Lazard \cite{la}.
\end{remark}
\begin{lemma} \label{le1.2}
  	Let $G, H$ be in the setting of Theorem \ref{thm-main} and assume $\T=\Z_p[[X]]$. Let $M$ be a finitely generated $\Z_p[[X]]{{}[}[G]]$ module which is finitely generated over $\Z_p[[X]][[H]]$. Also assume {{} there  exists} a height 1 prime ideal $Q_0$ in $\Z_p[[X]]$ with  $\Q_0 \neq (p)$, such that $\#(\frac{M}{Q_0M}\big)_U$ is finite for every open normal subgroup $U$ of $G$. Then for all but countably many height 1 prime ideals $Q$ in $\Z_p[[X]]$, $\#(\frac{M}{QM}\big)_U$ is finite for every open normal subgroup $U$ of $G$.
  \end{lemma}
  \begin{proof}
  	Let $U$ be an open normal subgroup of $G$. As $[G:U]<\infty$, $M_U$ is a finitely generated $\Z_p[[X]]$ module. Now, from the structure theorem for finitely generated $\Z_p[[X]]$ modules, we know that there exists a $\Z_p[[X]]$ module homomorphism:
  	\begin{small}{\begin{equation}	\label{e1}
  M_U \rightarrow \Z_p[[X]]^{r_U} {{} {\oplus} \bigoplus \limits_{i=1}^{s_{\tiny{U}}}} \dfrac{\Z_p[[X]]}{(Q_{i,U}^{n_i})}  
  	\end{equation}}\end{small}
   with finite kernel and cokernel, where $r_{\tiny{U}},n_i, s_{\tiny{U}} \in \N \cup \{0\}$ and $Q_{i,U}$'s are height 1 prime ideals in $\Z_p[[X]]$. Note that $ { ( \frac{M}{Q_0M})}_{U}\cong \frac {M_U}{Q_0M_U}.$
  	Hence using the fact that the cardinality of $(\frac {M}{Q_0M})_{U}$ is finite, we deduce from  \eqref{e1}  that $r_U=0$. So, if we choose $Q \in C$ such that $Q\not\in C_{M,U}:= {{} \bigcup\limits_{i=1}^{s_U}Q_{i,U}}$, then $\#\big(\frac{M}{QM}\big)_U$ is finite. Since $G$ is a profinite group, it has countable base at identity. Thus we can take the set $\mathcal U$ of open normal subgroups $U$ of $G$ to be countable. Then for any $Q \in C$ which does not belong to the countable set  $C_M:= \underset{U \in \mathcal U}{\cup} C_{M,U}$,  we get $\#\big(\frac{M}{QM}\big)_U$ is finite for every $U.$ This completes the proof of the lemma.
  \end{proof}
We next prove the following proposition:  
  \begin{proposition} \label{th1.1}
  	Let $G, H$ be {{} as} in  the setting of Theorem   \ref{thm-main}. Let $M$ be a finitely generated $\Z_p[[X]][G]]$ module which is finitely generated over $\Z_p[[X]][[H]]$. Then there exists  a countable subset $S_M$ of $S$ and further, for each $\theta \in S\setminus S_M$, there exists a countable subset $C_{M, \theta} $ of $C$  such the  the following holds: If we choose any $\theta \in S \setminus S_M$ and   any   $Q  \in C \setminus C_{M, \theta}$, then $\#\big(\frac{M}{QM}(\theta)  \big)_U$ is finite for every open normal subgroup $U$ of $G$.
    \end{proposition}
    \begin{proof}
    	Note that $\Z_p[[X]][[G]]\cong \Z_p[[G_1\times G]]$, where $G_1\cong \mathbb{Z}_{p}$. Let $U$ be any open normal subgroup of $G$. Then $U':=G_1 \times U$ is an open normal subgroup of $G_1 \times G$ that maps onto $U$ under the projection map $\pi:G_1 \times G \longrightarrow G$.
    	 Now by \ref{thm2}, there is a countable subset $S_M$ of $S$, such that for any  $\theta \in S \setminus S_M$, we have $(M(\theta))_{U'}= (M(\theta))_{G_1\times U}$ is finite for every $U$. 
        Therefore, $(M(\theta))_{U'}=(M(\theta)_{U})_{U'/U}\cong \frac{M(\theta)_{U}}{(X)}$ is finite for every $U$.

    	 So writing $N=M(\theta)$ and $Q_0=(X)$, we have  $\big(\frac {N}{Q_0 N}\big)_{U}$ is finite for every $ U$ open normal in $G$. Now applying  Lemma \ref{le1.2}, there is a countable subset $C_{M,\theta} \subset C$ such that for any $Q \in C\setminus C_{M, \theta}$, we have $(\frac {N}{Q N})_{U}= \big(\frac {M}{QM}(\theta)\big)_{U}$ is finite for every open normal subgroup $U$.
     \end{proof}

The next lemma is easy to prove and used later.
 \begin{lemma}\label{lem:1.8}
	Let $G, H$ be as in Theorem \ref{thm-main}. Let $G_1$ be a $p$-adic Lie group isomorphic to $\Z_p$.
	Let $K$ be an open subgroup of $G_1\times G$ such that $G_1 \times H \subset K$. Then $K=G_1\times G^{0}$, where $G^{0}$ is an open subgroup of $G$. \qed
\end{lemma}

 Using Lemma \ref{lem:1.8}, we deduce the following result.
\begin{lemma} \label{le1.4}
	Let $G$ be a compact, pro-$p$, $p$-adic Lie group without any $p$-torsion element. Let $H$ be a closed normal subgroup of $G$ with $\Gamma:=G/H \cong \Z_p$.  Let $M$ be a finitely generated $\Z_p[[X]][[G]]$ module which  is also  finitely generated  over $\Z_p[[X]][[H]]$. Then, there exists an open subgroup $G^0$ of $G$ with $H\subset G^0$ and a resolution 
	$$0\rightarrow N_k \rightarrow N_{k-1}\rightarrow \cdots \rightarrow N_1 \rightarrow M \rightarrow 0$$
	of $M$ by finitely generated $\Z_p[[X]][[G^0]]$ modules $N_i$,  $\ 1 \leq i \leq k$ such that each $N_i$ is a free $\Z_p[[X]][[H]]$ module of finite rank.
\end{lemma} 
\begin{proof}
	Note that $\Z_p[[X]][[G]]\cong \Z_p[[G_1 \times G]]$, where $G_1\cong \Z_p$. Also $\frac{G_1\times G}{G_1\times H} \cong \Gamma$. Using \cite[Lemma 2.4]{ss}, {{} there  exists} an open subgroup $G^{00}$ of $G_1\times G$ such that $G_1\times H \subset G^{00}$ and  a resolution
	$$0\rightarrow N_k \rightarrow N_{k-1}\rightarrow \cdots \rightarrow N_1 \rightarrow M \rightarrow 0$$ of $M$ by finitely generated $\Z_p[[G^{00}]]$ modules $N_i$, $1 \leq i \leq k$ such that each $N_i$ is a free $\Z_p[[G_1\times H]]$ module of finite rank. By Lemma \ref{lem:1.8}, $G^{00}= G_1\times G^{0}$, where $G^0$ is an open subgroup of $G$ containing $H$.
\end{proof}	
{{} 
\begin{remark}
	Let us again consider  the false Tate curve extension given by $K =\Q(\mu_p)$,  $K_\infty =\underset{n}{\cup}\Q(\mu_{p^\infty})( m^{1/p^n})$ and $G=\mathrm{Gal}(K_\infty/K) \cong  \Z_p\rtimes \Z_p$. Then it  follows from a result of Z\'abr\'adi \cite[Lemma 4.3]{z}, that for this particular $G$, one can, in fact take $G^0=G$ in \cite[Lemma 2.4]{ss} and hence we can also assume  $G^0$ to be equal to $G$ in Lemma \ref{le1.4}. 
\end{remark}
}
Next we calculate the Euler characteristic of a free $\Z_p[[X]][[H]] $ module. Recall, for any $Q\in C^{\Z_p[[X]]}$, $\frac{\Z_p[[X]]}{Q}\cong O_Q$, the ring of integers of certain finite extension of $\Q_p$.
\begin{proposition} \label{p1.7}
	Let $G$ be a compact, pro-$p$, $p$-adic Lie group without any $p$-torsion element. Let $H$ be a closed normal subgroup of $G$ with $\Gamma:=G/H \cong \Z_p$. 
	Let $N$ be a finitely generated $\Z_p[[X]][[G]]$ module which is also a finitely generated free $\Z_p[[X]][[H]] $ module of rank $d$. Then there exists  a countable subset $S_N$ of $S$ and for any  $\theta \in S\setminus S_N$ there exists a countable subset $C_{N, \theta}$ of $C$ such that the following holds:  For any  $\theta \in S\setminus S_N$ and for any $Q\in C\setminus C_{N, \theta}$, the $U$-Euler characteristic $\chi\left(U,\frac{N}{QN}(\theta)\right)$ exists for every open normal subgroup $U$ of $G$. Moreover, $\chi\big(U,\frac{N}{QN}(\theta)\big)=\# \big(\frac{N} {QN}(\theta)\big)_U$.
\end{proposition}
\proof By Proposition \ref{th1.1}, {{} there  exists}  a countable subset $S_N$ of $S$ and  for any  $\theta \in S\setminus S_N$ there exists a countable subset $C_{N, \theta}$ of $C$ such that the following holds: For any  $\theta \in S\setminus S_N$ and for any $Q\in C\setminus C_{N, \theta}$,  $H_0\left(U,\frac{N}{QN}(\theta)\right)$ is finite for every open normal subgroup $U$ of $G$.  Now as $\frac {N}{QN}$ is a free $O_Q[[H]]$ module of finite rank, the result follows directly from \cite[Proposition 2.7]{ss}. \qed

\begin{remark}\label{lastint}
Let us keep the setting and hypotheses of Lemma \ref{le1.4}. Then by Lemma \ref{le1.4}, there exists  an open normal subgroup $H \subset G^0 \subset G$  a resolution of $M$ by $\Lambda(G^0)$-module, 
\begin{equation} \label{eq6}
0\longrightarrow N_k\overset{f_k} \longrightarrow N_{k-1}\overset{f_{k-1}}\longrightarrow \cdots\overset{f_2}\longrightarrow N_1\overset{f_1} \longrightarrow M \rightarrow 0
\end{equation}
such that each $N_i$ is a finitely generated, free $\Z_p[[X]][[H]]$-module. Let $U$ be an open normal subgroup of $G^0$. Take $Q=(X+1)-\zeta (1+p)^k \in A_{\text{arith}} (\Z_p[[X]])$. By Corollary \ref{col11}, we can choose a $\theta \in S$ such that for every $Q \in A_{\text{arith}} (\Z_p[[X]])$, $\chi(U,\frac{N_i}{QN_i}(\theta))$, {{}$1\leq i \leq k$} and $\chi(U,\frac{M}{QM}(\theta))$ are all finite. An easy computation shows that for every $Q \in A_{\text{arith}} (\Z_p[[X]])$, the $U$-Euler characteristics  $M(\theta)[Q]$ also exists and  $\chi(U,\frac{M}{QM}(\theta))/\chi(U,M(\theta)[Q]) = \prod_i \big(\chi(U,\frac{N_i}{QN_i}(\theta))\big)^{(-1)^{i+1}}$.

 We also have $\chi(U,\frac{N_i}{QN_i}(\theta))=\#\big(\frac{N_i}{QN_i}(\theta)\big)_U=\# \frac{\big(N_i(\theta)\big)_U}{Q\big(N_i(\theta)\big)_U}$, for every $Q \in A_{\text{arith}} (\Z_p[[X]])$. 
Let $f^\theta_i(X)\in \Z_p[[X]]$ denote the $\Z_p[[X]]$ characteristic element of  $\big(N_i(\theta)\big)_U$. Since $\frac{\big(N_i(\theta)\big)_U}{Q\big(N_i(\theta)\big)_U}$ is finite for every $i$, $Q$ does not divide $f^\theta_i(X)$ and we have $\chi(U,\frac{N_i}{QN_i}(\theta))=f_i^\theta(Q):=f_i^\theta(\zeta (1+p)^k-1)=\#\frac{\Z_p[[X]]}{(Q,f^\theta_i)}$, for every $Q \in A_{\text{arith}} (\Z_p[[X]])$. Thus 
\begin{equation}\label{lasteqn}
\frac{\chi(U,\frac{M}{QM}(\theta))}{\chi(U,M(\theta)[Q])} = \prod_i \Big({\#\frac{\Z_p[[X]]}{(Q,f^\theta_i(X))}\Big)^{(-1)^{i+1}}}.
\end{equation}

Now in this setting, if for some $Q$, $M[Q]=0$, then using \eqref{lasteqn} we can compute $\chi(U,\frac{M}{QM}(\theta))$. On the other hand, it is also natural to ask if $\chi(U,\frac{M}{{{} Q_0}M}(\theta))=\chi(U,M(\theta)[{{} Q_0}])$ for some {{} $Q_0\in  A_{\text{arith}} (\Z_p[[X]])$}, then can we say $\chi(U,\frac{M}{QM}(\theta))=\chi(U,M(\theta)[Q])$ for  every $Q \in A_{\text{arith}} (\Z_p[[X]])$?  Unfortunately, we do not  know the answer of this. The reason being,  we do not know if $\prod_i \big(f_i^\theta(X))^{(-1)^{i+1}} \in \Z_p[[X]]$ or not.
\end{remark}

\begin{remark}\label{sfinal-remark}
Let $E/K$ be an elliptic curve with good, ordinary reduction at primes of $K$ dividing $p$. Let $p$-cohomological dimension $cd_p(G)$ of $G=\text{Gal}(K_\infty/K)$ be $\geq 3$. Also assume (i) for {{} any} prime ${{}v}$ of $K$ dividing $p$,  $cd_p(G_{{}v}) < cd_p(G)$ and (ii) for  {{} any} prime $u$ of $K$ not dividing $p$, such that either $u$ ramifies in $K_\infty$ or $u$ is a bad prime of $E$,  $cd_p(G_u) \geq 2 $. Further assume, $X(E/K_\cyc)$ is a finitely generated $\Z_p$ module. Then by  \cite[Proposition 5.2]{ov}, the $\Z_p[[G]]$ projective dimension of $X(E/K_\infty)$ is  $= cd_p(G) -1$ if $E(K)[p]\neq 0$ and $ = cd_p(G) -2$ if $E(K)[p]= 0$. Now, as $X(E/K_\cyc)$ is a finitely generated $\Z_p$ module, applying a control theorem, we can deduce  $X(E/K_\infty)$ is a finitely generated $\Z_p[[H]]$ module. Since   $cd_p(H) = cd_p(G) -1$, $\Z_p[[H]]$ projective dimension of $X(E/K_\infty)$ is  $ cd_p(H) -1$ if $E(K)[p]\neq 0$ and $ =cd_p(H) -2$ if $E(K)[p]= 0$. In particular, if  $cd_p(H) \geq 3 $, then $X(E/K_\infty)$ cannot be a free $\Z_p[[H]]$ module. Nevertheless, {{} there  exists} an open subgroup $G^0$ of $G$ containing $H$ and  a $\Z_p[[G^0]]$ resolution of $X(E/K_\infty)$  of length $k$ by finitely generated $\Z_p[[H]]$ modules \cite[Lemma 2.4]{ss}. Moreover,  the length of the resolution is given by $k= cd_p(H) -1$ if $E(K)[p]\neq 0$ and $k= cd_p(H) -2$ if $E(K)[p]=0$. A similar assertion holds for the `big' Selmer  group $\mathcal X(\T_\F/K_\infty)$ by using Lemma \ref{le1.4}.
\end{remark}

\medskip

Now we are ready to prove Theorem \ref{thm-main}.

\medskip

{\it Proof of Theorem \ref{thm-main}:} First of all, using Corollary \ref{hidaredcor}, we can assume without any loss of generality, that $\T=\Z_p[[X]]$. Next, by Remark \ref{rem101}, we can assume $M$ is a finitely generated $\Z_p[[X]][[H]]$ module. Further, as $\Z_p[[X]]$ is regular, we can assume using Corollary \ref{col1} that $M$ is $\T$-torsion free. Moreover, following Remark \ref{rem102}, we will assume that $G$ is a compact, pro-$p$, $p$-adic Lie group without any element of order $p$.

By Lemma \ref{le1.4}, there exists an open normal subgroup $G^0$ of $G$ with $H \subset G^0$ and a resolution 
\begin{equation} \label{eq2}
0\longrightarrow N_k\overset{f_k} \longrightarrow N_{k-1}\overset{f_{k-1}}\longrightarrow \cdots\overset{f_2}\longrightarrow N_1\overset{f_1} \longrightarrow M \rightarrow 0
\end{equation}
of $M$ by finitely generated $\Z_p[[X]][[G^0]]$ modules $N_i$, $1\leq i \leq k$ such that each $N_i$ is a free $\Z_p[[H]]$ module of finite rank. 

\par Set $\Gamma^0:= G^0/H$. Then as explained in the proof of \cite[Theorem 1.2]{ss}, for any given $\theta \in S, \ Q\in C$ and any open normal subgroup $U$ of $G$, if $\chi(U\cap G^0, \frac{M}{QM}(\theta|_{\Gamma^0}))$ is finite, then we can deduce $\chi(U, \frac{M}{QM}(\theta))$ is also finite.  Thus, for the rest of the proof, we will only discuss the finiteness of $\chi(U\cap G^0, \frac{M}{QM}(\theta|_{\Gamma^0}))$; and further, to ease the burden of notation, will write $G^0=G$ and $\Gamma^0=\Gamma$ for the rest of the proof.

  We will proceed by induction on $k$ in \eqref{eq2}. For $k=1$, $M/QM$ is a free $\frac{\Z_p[[X]]}{Q}[[H]]\cong O_Q[[H]]$ module  of finite rank, for any $Q\in C$. Hence by Proposition \ref{p1.7}, there exists a countable subset $S_M$ of $S$ and a countable subset $C_{M, \theta}$ of $C$, such that the following holds: For any  $\theta \in S\setminus S_M$ and for any $Q\in C\setminus C_{M, \theta}$,  $\chi\left(U,\frac{N}{QN}(\theta)\right)$ exists for every open normal subgroup $U$ of $G$. 
  
 Next, pick any $Q \in C$. Then $\Z_p[[X]]$ being a regular local ring, we get $Q=(q)$ and hence $M[q] =0$ as $M$ is $\T$-torsion free. Thus \eqref{eq2} gives rise to another exact sequence of ${O_Q}[[G]]$ modules
\begin{equation} \label{eq23}
0\lra\text{Img}({f_2})/{Q\text{Img}({f_2})}\longrightarrow N_1/{QN_1} \overset{f_1}\longrightarrow M/{QM} \rightarrow 0.
\end{equation}
 
 By induction, there exists a countable subset $S_2$ of $S$ and a countable subset $C_{2, \theta}$ of $C$, such that the following holds: For any  $\theta \in S\setminus S_2$ and for any $Q\in C\setminus C_{2, \theta}$,  $\chi\left(U,\frac{Img(f_2)}{QImg(f_2)}(\theta)\right)$ exists for every open normal subgroup $U$ of $G$. Similarly,  $N_1/{QN_1}$ is a free $O_Q[[H]]$ module  of finite rank and hence  there exists a countable subset $S_1$ of $S$ and a countable subset $C_{1, \theta}$ of $C$, such that the following holds: For any  $\theta \in S\setminus S_1$ and for any $Q\in C\setminus C_{1, \theta}$,  $\chi\left(U,\frac{N_1}{QN_1}(\theta)\right)$ exists for every open normal subgroup $U$ of $G$. Define $S_M: = S_1 \cup S_2$, and for any $\theta \in S\setminus S_M$, set $C_{M,\theta}=C_{1, \theta} \cup C_{2, \theta}$. Then from \eqref{eq23}, for any $\theta \in S\setminus S_M$ and any $Q\in C \setminus C_{M, \theta}$, $\chi\big(U, \frac{M}{QM}(\theta)\big)$ is finite for every $U$. This completes the proof of Theorem \ref{thm-main}. \qed

\section*{Acknowledgement} \begin{small}{S. Jha  acknowledges the support of SERB MATRICS grant and SERB ECR grant. S. Shekhar is supported by  DST INSPIRE faculty award grant. We thank Tadashi Ochiai for discussions. We thank the referee for her/his valuable comments and suggestions, which helped us in improving the article.}\end{small}

\end{document}